\documentclass[a4paper,11pt]{amsart}
\usepackage{amsmath,amssymb,amsthm}
\usepackage{graphics,color}
\usepackage[dvips]{graphicx}
\usepackage{epsfig}

\graphicspath{{Fig/}}

\newtheorem{dfn}{Definition}[section]
\newtheorem{thm}[dfn]{Theorem}
\newtheorem{prp}[dfn]{Proposition}
\newtheorem{lem}[dfn]{Lemma}
\newtheorem{cor}[dfn]{Corollary}

\theoremstyle{definition}
\newtheorem{exl}[dfn]{Example}
\newtheorem{rem}[dfn]{Remark}

\def\R{{\mathbb R}}
\def\Z{{\mathbb Z}}

\def\H{{\mathbb H}}

\def\phi{\varphi}

\title{Hyperbolization of cusps with convex boundary}

\date{December 20, 2015 \\ v2}

\author{Fran\c{c}ois Fillastre}
\address{Universit\'e de Cergy-Pontoise, UMR CNRS 8088, F-95000 Cergy-Pontoise, France}
\email{francois.fillastre@u-cergy.fr}
\author{Ivan Izmestiev}
\address{Department of Mathematics,
Free University of Berlin,
Arnimallee 2, D-14195 Berlin, Germany}
\email{izmestiev@math.fu-berlin.de}
\author{Giona Veronelli}
\address{Universit\'e Paris 13, Sorbonne Paris Cit\'e, LAGA, CNRS ( UMR 7539), 99 av. Jean-Baptiste Cl\'ement,
93430 Villetaneuse, France}
\email{veronelli@math.univ-paris13.fr}

\begin{document}

\begin{abstract}
We prove that for every metric on the torus with curvature bounded from below by $-1$ in the sense of Alexandrov there exists a hyperbolic cusp with convex boundary such that the induced metric on the boundary is the  given metric.
The proof is by polyhedral approximation.

This was the last open case of a general theorem: every metric with curvature bounded from below on a compact surface is isometric to a convex surface in a $3$-dimensional space form.
\end{abstract}

\maketitle

\tableofcontents

\section{Introduction}

\subsection{Statement of the results}

Let $T$ denote the $2$-dimensional torus.
A hyperbolic cusp $\mathrm{C}$ with convex boundary is a complete hyperbolic manifold of finite volume, homeomorphic to
$T\times [0,+\infty[$, and such that the boundary $\partial \mathrm{C}=T\times \{0\}$ is geodesically convex.
By the Buyalo convex hypersurface theorem \cite{AKP}, the induced inner metric
on $\partial \mathrm{C}$ has curvature bounded from below by $-1$ in the sense of Alexandrov --- in short, the metric is CBB($-1$), see 
Section~\ref{sec:cbb}. In the present paper, we prove that all the CBB($-1$) metrics on the torus are obtained in this way.

\begin{thm}\label{thm:main}
Let $m$ be a CBB($-1$) metric on the torus. Then 
there exists a hyperbolic cusp $\mathrm{C}$ with convex boundary such that the induced metric on $\partial \mathrm{C}$
 is isometric to $m$.
\end{thm}

Throughout the paper, by  induced metric we mean an intrinsic metric.
Examples of CBB($-1$) metrics on the torus are distances defined by Riemannian metrics of curvature $\geq -1$.
Using classical regularity results by Pogorelov (see below), Theorem~\ref{thm:main} implies the following. 

\begin{thm}\label{thm:cusp lisse}
Let $g$ be a smooth Riemannian metric of curvature $> -1$ on the torus.
Then  
there exists a hyperbolic cusp $\mathrm{C}$ with smooth strictly convex boundary such that the induced Riemannian metric on
$\partial \mathrm{C}$ is isometric to $g$. 
\end{thm}

Another examples of  CBB($-1$) metrics on the torus are  hyperbolic metrics with conical singularities of positive curvatures. Recall that the (singular) curvature at a cone singularity is $2\pi$ minus the total angle around the singularity. The proof of  Theorem~\ref{thm:main} will be done by polyhedral approximation, using the following result.

\begin{thm}[{\cite{FI}}]\label{thm:main1poly}
Let $m$ be a hyperbolic  metric with conical singularities of positive curvature on the torus. Then 
there exists a hyperbolic cusp $\mathrm{C}$ with convex polyhedral boundary such that the induced inner metric on
$\partial \mathrm{C}$ is isometric to $m$.
\end{thm}

Actually it is proved in \cite{FI} that the cusp $\mathrm{C}$ in the statement of Theorem~\ref{thm:main1poly} is unique.
One can hope a uniqueness result in Theorem~\ref{thm:main} (that would also imply uniqueness in Theorem~\ref{thm:cusp lisse}). This should be  the subject of a forthcoming paper.

The above statements were the last steps in order to get the following general statement.

\begin{thm}\label{thm:general cbb}
Let $m$ be a CBB($k$) metric on a compact surface. Then $m$ is isometric to a convex surface $S$ in a Riemannian space of constant curvature $k$.
Moreover
\begin{itemize}
\item if $m$ is a  metric of curvature $k$ with conical singularities of positive curvature, 
then $S$ is polyhedral,
\item if $m$ comes from a smooth Riemannian metric with curvature $>k$, then 
$S$  is smooth and strictly convex.
\end{itemize}
\end{thm}

In Section~\ref{1.2} we will recall all the results leading to Theorem~\ref{thm:general cbb}. 
The proof of Theorem~\ref{thm:main} is based on a general result of polyhedral approximation that is recalled in Section~\ref{sec:cbb} (Theorem~\ref{thm: aprox pol}). Going to the universal cover, 
boundaries of convex hyperbolic cusps are seen as convex surfaces of the hyperbolic space invariant under the action of a group of isometries acting cocompactly on a horosphere. Such surfaces are graphs of \emph{horoconvex functions} defined on the horosphere. They are introduced in Section~\ref{sub:horo}. They can be written in terms of convex functions on the plane, hence they will inherit strong properties from convex functions.

Then we prove Theorem~\ref{thm:main} in Section~\ref{sec:proof}.
Theorem~\ref{thm:main1poly} and Theorem~\ref{thm: aprox pol} give a sequence of polyhedral surfaces in $\H^ 3$. One shows that this sequence converges to a convex surface, invariant under the action of a group $\Gamma$. The main point is to check that the induced metric on the quotient of the surface by $\Gamma$ is isometric to the metric $m$ (Section~\ref{sec:conv met}).

\subsection{Hyperbolization of products manifolds}\label{1.2}

The notion of metric space of non-negative curvature  was introduced by A.D.~Alexandrov
in order to describe  the induced metric on the boundary of
convex bodies in $\R^3$. He proved that this property characterizes 
the convex bodies in the sense described below. Here we list several theorems of existence.
\begin{thm}[{\cite{al}}]\label{thm:cbb0}
Any CBB($0$) metric on the sphere is isometric to the boundary of a convex body in $\R^3$.

\end{thm}

The hyperbolic version is as follows.

\begin{thm}[{\cite{al}}]\label{thm:ale1}        
Any CBB($-1$) metric on the sphere is isometric to the boundary of a compact convex set in $\H^3$.

\end{thm}

Let us mention the following general local result, obtained 
from Theorem~\ref{thm:ale1}  and a gluing theorem.

\begin{thm}[{\cite{al}}]
Each point on a CBB($-1$) surface has a neighbourhood isometric to a convex surface in $\H^3$.
\end{thm}

Forgetting the part of the hyperbolic space outside  the convex body, one derives from Theorem~\ref{thm:ale1} the following hyperbolization theorem for the ball.

\begin{thm}\label{thm:ale2}
Let $m$ be a CBB($-1$) metric on the sphere. Then 
there exists a hyperbolic ball $M$ with convex boundary  such that the induced  metric on
$\partial M$ is isometric to $m$.
\end{thm}

Actually Theorem~\ref{thm:ale2} also implies Theorem~\ref{thm:ale1}, because in this case the developing map is an 
isometric embedding \cite[Proposition I.1.4.2.]{Notesonnotes}.  Theorem~\ref{thm:ale1}
and  Theorem~\ref{thm:ale2} are proved by polyhedral approximation. For example,
Theorem~\ref{thm:ale1} is proved from the following particular case.

\begin{thm}\label{thm:alexpoly}
Let $m$ be a hyperbolic metric with conical singularities of positive curvature on the sphere. Then 
there exists a hyperbolic ball $M$ with convex polyhedral boundary  such that the induced  metric on
$\partial M$ is isometric to $m$.
\end{thm}

The following regularity result roughly says that if the metric on a convex surface in $\H^3$ is smooth, then
the surface is smooth.

\begin{thm}[{\cite[Theorem 1, chap. V §8]{pog}}]\label{thm:pogo}
Let $S$ be a surface with a $C^k$, $k\geq 5$, Riemannian metric of 
curvature $>-1$. 
If $S$ admits a convex isometric embedding into $\H^3$, 
then its image is a $C^{k-1}$ surface.  
\end{thm}

See \cite{CX15} for more precise results if  $S$ is homeomoprhic to the sphere, in particular if the curvature is $\geq -1$.
Theorem~\ref{thm:pogo} and Theorem~\ref{thm:ale1} immediately give the following.

\begin{thm}\label{thm:alexlisse}
Let $g$ be a smooth Riemannian metric with  curvature $>-1$ on a the sphere. Then there exists a  hyperbolic ball $M$ with smooth convex boundary such that 
the induced metric on $\partial M$ is isometric to $g$.
\end{thm}

For metrics on a compact (connected) surface $S$ of genus $>1$, the following result was recently proved.

\begin{thm}[{\cite{slu}}]\label{thm:dima}
 Let $M$ be a compact connected 3-manifold with boundary of the
type $S \times [-1, 1]$. Let $m$ be
a CBB($-1$) metric on $\partial M$. Then there exists a
hyperbolic metric in $M$ with a convex boundary such that the  induced  metric
on $\partial M$ is isometric to $m$.
\end{thm} 

The proof of Theorem~\ref{thm:dima} goes by smooth approximation. The smooth 
version of Theorem~\ref{thm:dima} is included in the following more general result.

\begin{thm}[{\cite{lab}}]\label{thm:lab}
Let $M$ be a compact manifold with boundary (different from the
solid torus) which admits a structure of a strictly convex hyperbolic manifold. Let
$g$ be a smooth  metric on $\partial M$ of  curvature $>-1$. Then there exists a convex hyperbolic metric  on $M$ which
induces $g$ on $\partial M$.
\end{thm}

See also \cite{sch}, which  contains
a uniqueness result.
Of course, one can take for metrics $m$ in the statement of Theorem~\ref{thm:dima} hyperbolic metrics on
$S$ with conical singularities of positive curvature \cite{slu2}. But the boundary of the solution is not necessarily of
polyhedral type. This is because the boundary may meet the  boundary of the convex core of $M$.
If $M$ is \emph{Fuchsian}, that is if its convex core is a totally geodesic surface,  
this cannot happen. Actually we have the following.

\begin{thm}[{\cite{f}}]\label{thm:f}
Let $m$ be a hyperbolic metric with conical singularities of positive curvature on a compact surface $S$ of 
genus $>1$. Then there exists a Fuchsian hyperbolic manifold $S\times[-1,1]$ with polyhedral convex boundary such that 
the induced metric on the boundary components $S\times \{-1\}$ and 
$S\times \{1\}$ are isometric to $m$.
\end{thm}

The smooth analogue was known for a long:

\begin{thm}[{\cite{gro}}]\label{thm:gromov}
Let $g$ be a smooth Riemannian metric with  curvature $>-1$ on a compact surface $S$ of 
genus $>1$. Then there exists a Fuchsian hyperbolic manifold $S\times[-1,1]$ with smooth convex boundary such that 
the induced metric on the boundary components is isometric to $g$.
\end{thm}

Using Theorem~\ref{thm:gromov} instead of Theorem~\ref{thm:lab}, the proof of Theorem~\ref{thm:dima} leads to the following.

\begin{thm}\label{thm:fuchs general}
Let $m$ be a CBB($-1$)  metric on a compact surface $S$ of 
genus $>1$. Then there exists a Fuchsian hyperbolic manifold $S\times[-1,1]$ with convex boundary such that 
the induced metric on the boundary components is isometric to $m$.
\end{thm}

Let us put all these statements together.
Cutting in a suitable way the hyperbolic manifolds given in theorems \ref{thm:ale1}, \ref{thm:alexlisse}, \ref{thm:alexpoly},
\ref{thm:main},
\ref{thm:cusp lisse}, \ref{thm:main1poly}, \ref{thm:fuchs general}, \ref{thm:gromov}, \ref{thm:f},  we obtain the following result.

\begin{thm}\label{thm:general hyperbolic}
Let $m$ be a CBB($-1$) metric on a compact surface $S$. Then the manifold $S\times [-1,1]$ admits a hyperbolic metric, such that $S\times \{-1\}$ is convex and isometric to $m$ for the induced inner metric, and 
$S\times \{1\}$ has constant curvature.

Moreover
\begin{itemize}
\item if $m$ is a hyperbolic metric with conical singularities of positive curvature, 
then $S\times \{-1\}$ is polyhedral,
\item if $m$ comes from a smooth Riemannian metric with curvature $>-1$, then 
$S\times \{-1\}$  is smooth and strictly convex.
\end{itemize}
\end{thm}

Note that in the case of genus $>1$, we have chosen the Fuchsian solution, but the quasi-Fuchsian Theorem~\ref{thm:dima} gives many choices for the realization of the prescribed metric. 
Actually all the cases in Theorem~\ref{thm:general hyperbolic}
share the same property: the holonomy of their developing map
fixes a point (the point may not be in the hyperboic space, see for example the beginning of Section~\ref{sec:proof}).

In the case 
of CBB($-1$) metrics on the torus, we could also consider a hyperbolic metric with convex boundary on a full torus.
In this direction, only the smooth case is known.

\begin{thm}[{\cite{sch}}]
Let $g$ be a smooth Riemannian metric of curvature $>-1$ on the torus $T$.
Then there exists a (unique) hyperbolic metric on the full torus such that 
the metric on the boundary is smooth, strictly convex and isometric to $g$.
\end{thm}

Another question is to realize CBB($-1$) metrics on compact surfaces of genus $>1$ as the convex boundary of 
more general compact hyperbolic manifold, analogously to Theorem~\ref{thm:lab}.

We cited Theorem~\ref{thm:cbb0} about realization of CBB($0$) metrics on the sphere in the Euclidean space.
There is also an analogue result about  realization of CBB($1$) metrics on the sphere in the $3$ dimensional sphere \cite{al}, as well as the polyhedral and smooth counterparts. Theorem~\ref{thm:general hyperbolic} gives all the possibilities for 
a CBB($-1$) metric on a compact surface of genus $>1$. Moreover, it is obvious that any flat metric (i.e. a metric of curvature $0$ everywhere) on a torus $T$ can be extended to a flat metric on $T\times [-1,1]$.
Lemma~\ref{lem:GB} says that we have exhausted all the possibilities. Theorem~\ref{thm:general cbb} follows.

A question is to know if the constant curvature metric on $S\times [-1,1]$ is unique.
Due to the work of Pogorelov, the answer is positive if $S$ is the sphere \cite{pog}.
As we already mentioned, in the torus case this is work in progress. As the only unsolved case there would remain that of Fuchsian hyperbolic manifolds with convex boundary.

\subsection{Smooth variational approach?}

As we said, Theorem~\ref{thm:cbb0} was proved by polyhedral approximation. It is based on the following seminal theorem, proved in the 1940's.

\begin{thm}[{\cite{al}}]\label{thm:al poly eucl}
Let $m$ be a flat metric with conical singularities of positive curvature on the sphere.
Then there exists a  convex polyhedron in Euclidean space with inner induced metric $m$ on the boundary.
\end{thm}

The proof of Theorem~\ref{thm:al poly eucl} is done by a continuity method, based on topological arguments,
in particular the Domain Invariance Theorem. Some years ago, a variational proof of  Theorem~\ref{thm:al poly eucl} was given in \cite{ivan1}. The functional is a discrete Hilbert--Einstein functional. It was then used in \cite{ivan2}, and later in \cite{FI} to prove Theorem~\ref{thm:main1poly}. A long-standing question is 
to use the smooth Hilbert--Einstein functional to give a variational proof of the smooth version of 
Theorem~\ref{thm:al poly eucl} (known as Weyl problem) \cite{BH,ivan3}. It would be interesting to give a variational proof of Theorem~\ref{thm:cusp lisse} as well.
There are reasons to think that the functional will have good properties 
in the case of a hyperbolic cusp.

\subsection{Acknowledgement}

The authors thank Stephanie Alexander, Thomas Richard, Jo{\"e}l Rouyer, Dima Slutskiy for useful conversations. 

The first author was partially supported by the  ANR GR-Analysis-Geometry and by the mathematic department of the Universidade Federal do Rio de Janeiro and this research has been conducted as part of the project Labex MME-DII (ANR11-LBX-0023-01). A part of this work was done during his stay at the Universidade Federal do Rio de Janeiro. He thanks this institution for its hospitality.

The second author was supported by by the European Research Council under the European Union's Seventh Framework Programme (FP7/2007-2013)/\allowbreak ERC Grant agreement no.~247029-SDModels.

The third author was partially supported by the
Gruppo
Nazionale per l'Analisi Matematica, la Probabilit\`a e le loro Applicazioni
(GNAMPA).

\section{Background}

\subsection{CBB metrics on compact surfaces}\label{sec:cbb}

We follow \cite{BBI} for basic definitions and results about metric geometry.
See also \cite{BH1999} and \cite{alexbook}.
Let $m$ be a metric on a compact surface $S$ (by this we imply that the topology
given by $m$ is the topology of $S$). We suppose that $m$ is intrinsic, that is for any $x,y\in S$, $m(x,y)$ is equal to the infimum of the length of the continuous curves between $x$ and $y$. By the Hopf--Rinow theorem, there always exists a shortest path between $x$ and $y$.

The metric $m$ is CBB($k$) if every point has a neighbourhood $U$ such that  any triangle contained in $U$ is thicker than the comparison triangle in the model space of constant curvature $k$ (see the references above for precise and equivalent definitions). By the Toponogov globalization theorem, this property is actually true for any triangle in $(S,m)$.

A shortest path between two points in a CBB($k$) space may not be unique, as show the example of a disc of curvature $k$ with a sector of angle $0<\alpha <2\pi$ removed and the two resulting sides identified. But shortest paths in CBB($k$) do not branch.

Let $(S,m)$ be a polyhedral CBB($k$) metric, that is a metric of constant curvature $k$ with singular curvatures  $k_i$ (the $k_i$ have to be positive \cite[10.9.5]{BBI}).
 It has to satisfy the Gauss--Bonnet formula
$$2\pi\chi(S) = k \operatorname{area}(S) +\sum k_i $$
i.e.
$$2\pi\chi(S) \geq k  \operatorname{area}(S)  $$
with equality if and only if $m$ is a smooth constant curvature metric (no conical singularities).

Now let $(S,m)$ be any CBB($k$) metric. By a theorem of Alexandrov and Zalgaller, $(S,m)$ can be decomposed into non-overlapping geodesic triangles. Replacing each triangle by a comparison triangle in the space of constant curvature $k$, we obtain a polyhedral CBB($k$) metric on $S$ see  \cite{richard,rou} for details. In particular, we obtain the following.

\begin{lem}\label{lem:GB}
A compact surface $S$ can be endowed with a CBB($k$) metric if and only if
\begin{itemize}
\item if $S$ is a sphere, $k\in \R$,
\item if $S$ is a torus, $k=0$ and the metric is a flat Riemannian metric or $k<0$,
\item otherwise, $k<0$.
\end{itemize}
\end{lem}

Actually, Alexandrov and Zalgaller proved much more, but in a different context.  Roughly speaking, the triangulation of
the CBB($-1$)  
can be chosen as fine as wanted. If the perimeter of the triangles goes to $0$, then the sequence of polyhedral metrics obtained by replacing the triangles by comparison triangles converge to $(S,m)$ in the Gromov--Haussdorff sense \cite{richard,rou}.

\begin{thm}\label{thm: aprox pol}
Let $m$ be a CBB($-1$) metric on a compact surface.
Then there exists polyhedral CBB($-1$) sequence of metric $m_n$ on the torus 
Gromov Hausdorff converging to $m$.
\end{thm}

At the end of the day, in the case of the torus, it is
not hard to conclude from Theorem~\ref{thm:main} and Proposition~\ref{prop:conv metric} that the convergence can be taken uniform in Theorem~\ref{thm: aprox pol}.

Let us mention the following results about Gromov--Hausdorff convergence that we will use in the sequel.

\begin{lem}[{\cite[I.5.40]{BH1999},\cite[7.3.14]{BBI}}]\label{lem:diam conv}
A Gromov--Hausdorff convergence of metric spaces $(S,m_n)$ implies the convergence of the diameters of $(S,m_n)$.
\end{lem}

\begin{thm}[{\cite{BGP},\cite[10.10.11]{BBI}}]
If a sequence $(S,m_n)$ of CBB($k$) metrics on a compact surface $S$ converges in the Gromov--Hausdorff sense to a CBB($k$) metric on $S$, then the sequence of the areas of $(S,m_n)$ (the total two dimensional
 Hausdorff  measure) is bounded from below by a positive constant.
\end{thm}

\subsection{Horoconvex functions}\label{sub:horo}

We identify $\R^2$ with a given horosphere $H\subset \H^3$, with center at $\infty$ (recall the definition of the Poincar\'e half-space model of $\H^3$).
We get coordinates $(x,t)$ on $\H^3=H\times \R$, where $t$ is the signed distance from a point to $H$: it is positive if and only if the point is in the exterior of the horoball bounded by $H$.
Note that it is the length of the segment between $x$ and its orthogonal projection onto $H$, and that the line from $x$ to $\infty$ is orthogonal to $H$.

Let $u:\R^2\rightarrow\R$. The \emph{horograph} of $u$ is the subset  $(x,u(x))\in \H^3$ for those coordinates. The horograph is said to be convex if the surface
is convex in $\H^3$ in the sense that it bounds a geodesically convex set. In the Klein projective model, this corresponds to the affine notion of convexity.

\begin{rem} 
In the upper half plane model, if the horosphere $H$ is the horizontal  plane at height $1$, then the horograph of $u$ is the graph of $e^{-u}$.
\end{rem}

We have the following characterization of horographs. It was already known in the smooth case \cite{GSS}. Let us also mention that the Darboux equation related to Theorem~\ref{thm:cusp lisse} is studied in \cite{RS}.

\begin{prp}\label{prop:car u}
The horograph of $u:\R^2\rightarrow \R$ is a convex surface if and only if the function 
$$x\mapsto e^{-2u(x)}+\|x\|^2$$
is convex, with $\|\cdot\|$ the Euclidean norm.
\end{prp}

In particular, $e^{-2u}$ is semi-convex, or lower-$C^{\infty}$, compare with
10.33 and 13.27 in \cite{RW}. We will call \emph{horoconvex} a function satisfying the
hypothesis of the proposition.

\begin{proof}
As above, consider coordinates $(x,t)_H$ on $\H^3=H\times \R$.
The horograph of $u$ is convex in $\H^3$ if and only if at each point $(x_0,u(x_0))_H$ there exists a totally geodesic surface $\Sigma$ containing the point $(x_0,u(x_0))_H$ and such that $\Sigma\subset\{(x,y)_H:y\geq u(x)\}$.
In the the Poincar\'e halfspace model let us now consider the standard Euclidean coordinates $(y,s)_E\in \R^2\times(0,\infty)$. Without loss of generality we can assume that the horosphere $H$ is the plane at height $1$ in this model. Then we have $(x,t)_H=(x,e^{-t})_E$. In this system, $\Sigma$ has to be a half-sphere with center $(c,0)_E$ on the plane at infinity $\R^2\times\{0\}$. In particular every such a half-sphere containing the point $(x_0,u(x_0))_H$ is given by $$\{(x,s)_E: e^{-2u(x_0)}+\|c-x_0\|^2=\|x-c\|^2+s^2\}~.$$
Coming back to $(,)_H$ coordinates, we have obtained that the horograph of $u$ is convex if and only if for all $x_0\in\R^2$ there exists a point $c\in\R^2$ such that for any $x\in\R^2$
$$
u(x)\leq -\frac 12 \ln \left(e^{-2u(x_0)}+\|c-x_0\|^2-\|c-x\|^2\right)~,$$
that is,
\begin{equation}\label{conv1}
e^{-2u(x)}-e^{-2u(x_0)}\geq\|c-x_0\|^2-\|c-x\|^2=-\|x-x_0\|^2+2\left\langle G,x-x_0\right\rangle~,
\end{equation}
where $G:=c-x_0\in\R^2$. Since $x\in\R^2$ is arbitrary, \eqref{conv1} is equivalent to
\begin{equation}\label{conv3}
e^{-2u(x_0+v)}+\|x_0+v\|^2-e^{-2u(x_0)}-\|x_0\|^2\geq \left\langle 2(x_0+G),v\right\rangle,\quad\forall v\in\R^2~.
\end{equation}
In turn, \eqref{conv3} means that if $G \in\R^2$ is such that \eqref{conv1} is satisfied, then at each point $x_0\in\R^2$ the graph of the function $e^{-2u(x)}+\|x\|^2$ has the planar graph of  $v\mapsto \left\langle 2(x_0+G),v\right\rangle+e^{-2u(x_0)}+\|x_0\|^2$ as a support plane. Hence $e^{-2u(x)}+\|x\|^2$ is a convex function on $\R^2$ if and only if the horograph of $u$ is convex in $\H^3$.
\end{proof}

\begin{rem}
Suppose that $x=x_0+sh$ for some unitary vector $h\in\R^2$ and $s>0$. Then \eqref{conv1} reads
\begin{equation}\label{conv2}
\frac 1s\left( e^{-2u(x_0+sh)}-e^{-2u(x_0)}\right)\geq-s|h|^2+2\left\langle G,h\right\rangle~.
\end{equation}
If $u\in C^1(\R^2)$, taking the limit as $s\to 0$ we get that
$$
\left\langle\textrm{grad}_{x_0}e^{-2u},h\right\rangle\geq \left\langle 2G,h\right\rangle~,
$$
and since the latter inequality holds for both $h$ and $-h$ we have necessarily that $G$ is unique and $$G=\frac{1}{2}\textrm{grad}_{x_0}e^{-2u}=-e^{-2u(x_0)}\textrm{grad}_{x_0}u~.
$$
\end{rem}

\begin{cor}\label{cor: ajout cst}
Let $u$ be horoconvex and $\epsilon >0$. Then 
$u+\epsilon$ is horoconvex.
\end{cor}
\begin{proof}
By Proposition~\ref{prop:car u}, one has to see that
$e^{-2\epsilon}e^{-2u(x)}+\|x\|^2$ is convex, 
that is equivalent to the convexity of $e^{-2u(x)}+e^{2\epsilon}\|x\|^2$. Let $f+g$ be convex, $g$ convex and $\lambda >1$.
Then $f+\lambda g = (f+ g) + (\lambda-1)g$ is convex as a sum of two convex functions.
\end{proof}

\begin{exl}
In  dimension $1$, the function $t\mapsto  \cos(t)/20$ is horoconvex, see Figure~\ref{fig:graph}.
More generally, any function $C^2$ close to a constant function is 
horoconvex. But $t\mapsto \cos(t)$ is not horoconvex. This example shows that $u$ horoconvex does not imply $\lambda u$ horoconvex.

\begin{figure}[h!]
\centering
\includegraphics[scale=0.5]{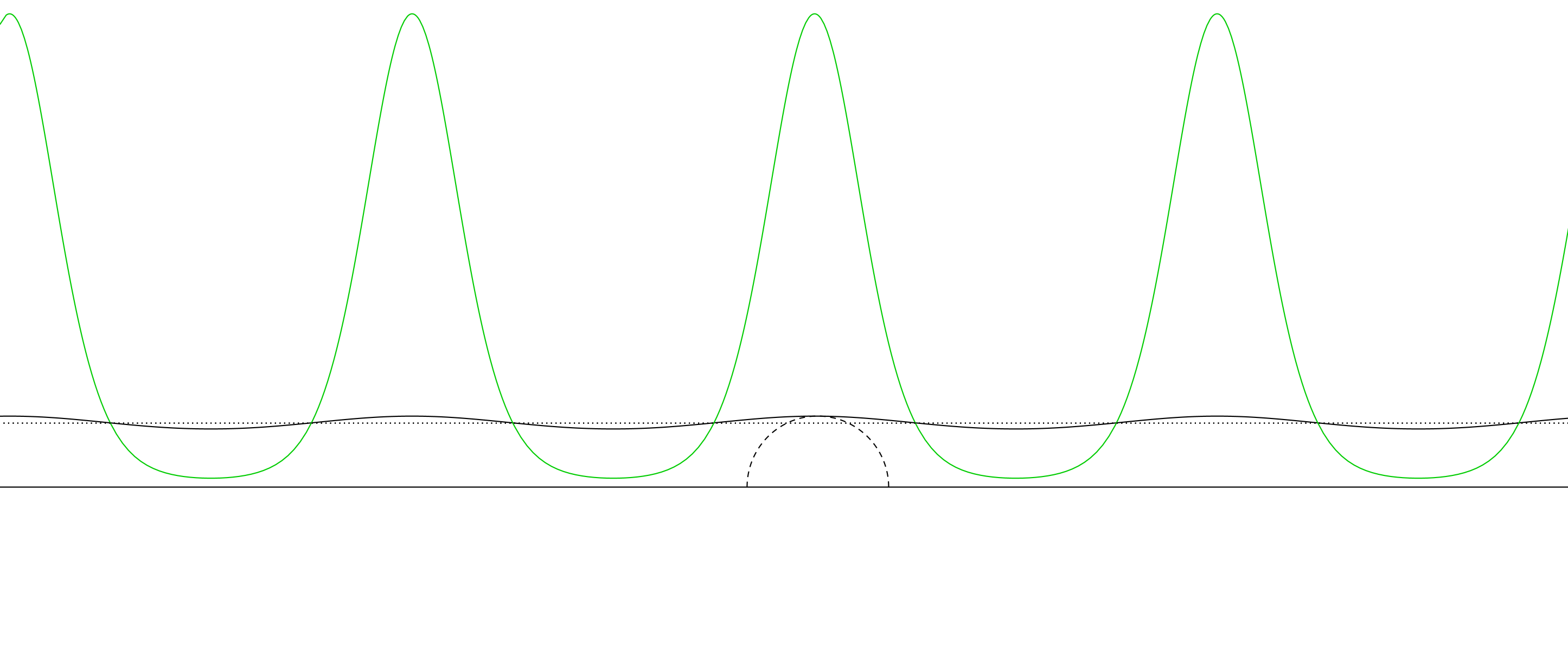}
\caption{Graphs of $t\mapsto e^{-\cos(t)/10}$ and $t\mapsto e^{-2\cos(t)}$.}
\label{fig:graph}
\end{figure}
\end{exl}

Horoconvex functions inherit strong properties from 
convex functions.

\begin{cor}\label{cor : equilip}
Any  sequence  of uniformly bounded  horoconvex functions is equi-Lipschitz on 
any compact set of $\R^2$.

Moreover, up to extracting a subsequence, the sequence converges uniformly on 
any compact set to a horoconvex function.
\end{cor}
\begin{proof}
Let $C\subset \R^2$ be a compact set and
$(u_n)$ a sequence of uniformly bounded  horoconvex functions.
Let $F_n(x)=e^{-2u_n(x)}+\|x\|^2$, which is convex by Proposition~\ref{prop:car u}. By
\cite[10.4]{Roc97}, there exists an $\epsilon >0$ such that for any $x,y\in C$
$$|F_n(x)-F_n(y)|\leq \frac{\operatorname{max}_C F_n -\operatorname{min}_C F_n}{\epsilon} \|x-y\|~. $$
As the $u_n$ are uniformly bounded, $F_n$ are uniformly bounded on $C$, hence there exists
a number $A$ satisfying
$$|F_n(x)-F_n(y)|\leq A \|x-y\|~. $$
Using again that the $u_n$ are uniformly bounded, and that $$\left|\|x\|^2-\|y\|^2\right|\leq2\left(\sup_C\|x\|\right)\|x-y\|~,$$
we thus obtain that 
$$|u_n(x)-u_n(y)|\leq A' \|x-y\| $$
for some $A'=A'(C)>0$ independent of $n$. So the sequence is equi-Lipschitz on $C$.

Up to extract a subsequence, the sequence of  convex functions $(F_n)$
converge (uniformly on compact sets) to a convex function $F$
 \cite[Theorem~10.9]{Roc97}). As the $u_n$ are uniformly bounded, there exists a positive constant $c$ such that $F_n\geq c +\|\cdot\|^2$, so $F>\|\cdot\|^2$, hence the function
$$u=-\frac{1}{2}\ln(F-\|\cdot\|^2) $$
is well defined and horoconvex by definition.
As $(F_n)$ converges  to $F$ uniformly on compact sets, it follows easily 
that $u_n$ converges to $u$ uniformly on compact sets.
\end{proof}

In particular, a horoconvex function is Lipschitz on any compact set.
By  Rademacher theorem, it is differentiable almost everywhere.

\subsection{Induced metric}

The length of a curve $c$ in $\H^3$ is the supremum of the length of all the polygonal paths of $\H^3$
with vertices on $c$. Equivalently \cite[Theorem~2.7.6]{BBI}, if $c$ is Lipschitz  (in $\H^3$), then

$$L(c)=\int \|c'\|_{\H^3}~.$$

Let $S$ be a convex surface in $\H^3$.  The (intrinsic) \emph{induced metric} on $S$ between two points 
$a,b$ of $S$ is the infimum of the lengths of all the rectifiable Lipschitz curves between $a$ and $b$. 

Let $u$ be a horoconvex function. 
Let $\tilde{d}_u$ be the intrinsic metric induced on the horograph of $u$. For simplicity, we will look at  metrics onto $\R^2$ rather than on the surfaces: for $(x,y)\in \R^2$, 
$$d_u(x,y):=\tilde{d}_u\left(\left(x,e^{-u(x)}\right),\left(y,e^{-u(y)}\right)\right)~. $$

Note that, in the upper half space model, notions of locally Lipschitz are equivalent for the metrics of $\H^3$ and $ \R^3$.
Also, the projection from $\R^3$ onto the horizontal plane is contracting. Hence, 
a locally Lipschitz curve of $\H^3$ is projected onto a locally Lipschitz curve of $\R^2$.

Let $c:[a,b]\rightarrow \R^2$ be a Lipschitz curve.
Let $c_u$ be the corresponding curve in the horograph of $u$ i.e. on the 
graph of $e^{-u}$:
$$c_u=\binom{c}{e^{-u\circ c}}~. $$
As $u$ is Lipschitz, $c_u$ is a Lipschitz curve of $\H^3$.

Let us denote by $L_u(c)$ the length of $c_u$ for the
metric $\tilde{d}_u$ on $S_u$. Using the half-space model metric
$$ L_u(c)=\int_a^b \frac{\|c_u'\|_{\R^3}}{(c_u)_3} =\int_a^b e^{u\circ c} \left(\|c'\|^2+((u\circ c)')^2e^{-2(u\circ c)} \right)^{1/2}$$

i.e.

\begin{equation}\label{eq:length}
L_u(c)= \int_a^b \left( e^{2u\circ c}\|c'\|^2+((u\circ c)')^2 \right)^{1/2}~.
\end{equation}

\begin{lem}\label{lem:bilip}
Let $(u_n)$ be a uniformly bounded sequence of horoconvex functions.  
Then on any compact set $K$, $d_{u_n}$ are uniformly Lipschitz equivalent to 
the Euclidean metric: $\exists \lambda_1,\lambda_2>0$ such that
$$\lambda_1 d_{\R^2} \leq d_{u_n} \leq \lambda_2 d_{\R^2}~. $$
Moreover, for any Lipschitz curve $c$ contained in $K$,
$$\lambda_1 L_{\R^2}(c) \leq L_{u_n}(c) \leq \lambda_2 L_{\R^2}(c)~. $$
\end{lem}
\begin{proof}
By construction, for any $x,y\in K$, $d_{u_n}(x,y)$ is 
not less than the distance in $\H^3$ between the corresponding points 
on the horograph of $u$.
As $K$ is compact and the $u_n$ uniformly bounded, the corresponding horographs 
above $K$ in $\R^3$ are contained in a hyperbolic ball. There exists a constant
$\lambda_1$ such that on this ball, $d_{\H^3}\geq \lambda_1 d_{\R^3}$. Also, with evidence, $d_{\R^3}\geq d_{\R^2}$. Hence,  $d_{\H^ 3}\geq \lambda_1 d_{\R^2}$. 
As the length of a Lipschitz curve for $d_{u_n}$ is the supremum 
of the length of  shortest polygonal paths \cite[2.3.4,2.4.3]{BBI} it follows that for any curve $c$ in $K$,
$\lambda_1 L_{\R^2}(c)\leq L_{u_n}(c)$.

On the other hand, for any curve $c$ in $K$, it follows from \eqref{eq:length} and from the inequality 
$\sqrt{a^2+b^2}\leq |a| + |b| $ that
$$L_{u_n}(c)\leq \int_a^b e^{-u_n}\|c'\| + |(u\circ c)'| $$
and
as the $u_n$ are uniformly bounded  on $K$ by assumption, 
and also $u_n$ are equi-Lipschitz on $K$ by Corollary~\ref{cor : equilip}, then there exists $\lambda_2$ such that  
$$L_{u_n}(c)\leq \lambda_2 L_{\R^2}(c)~.$$

\end{proof}

The proof of the following lemma mimics the one for convex bodies in Euclidean space \cite[p. 358]{BBI}.  

\begin{lem}\label{cor:BF2}
Let $u,v$ be horoconvex such that $u\leq v$ and
$$\delta= \operatorname{sup}_{\R^2}(v-u)$$
is finite.
Then
$$d_u\leq d_v+ 2 \delta~. $$
\end{lem}
\begin{proof}

Recall that the Hyperbolic Busemann--Feller Lemma, \cite[II.2.4]{BH1999} says that 
the orthogonal projection onto a convex set in $\H^3$ is contracting. 
 $u\leq v$ implies that the horograph of $v$ is in the exterior of the convex set bounded by the horograph of 
$u$, so the orthogonal projection is well defined from the horograph of $v$ onto the horograph of $u$.
For $a\in \R^2$, let $p_\bot(a)$ be the vertical projection onto
$\R^2$ of the orthogonal projection of $(a,v(a))$ onto the horograph of $u$.

On one hand, Busemann--Feller Lemma implies that
\begin{equation}\label{eqBF2}
 d_u(p_\bot(a),p_\bot(b))\leq d_v(a,b)~.
\end{equation}

On the other hand, Busemann--Feller Lemma implies that
$d_u(p_\bot(a),a)$ is less than the distance 
in $\H^3$ between $(a,v(a))$ and $(a,u(a))$, 
and this last quantity is less than $\delta$ by assumption, so
 
 \begin{equation}\label{eq:2}
  d_u(p_\bot(a),a) \leq \delta~.
 \end{equation}

 The result follows from \eqref{eqBF2}, \eqref{eq:2} and the triangle inequality.
\end{proof}

\section{Proof of Theorem~\ref{thm:main}}\label{sec:proof}

Now let $(S,m)$ be a CBB($-1$) metric on the torus. According to Theorem~\ref{thm: aprox pol} there exists a sequence of polyhedral CBB($-1$) metric $m_n$ on the torus 
Gromov--Hausdorff converging to $m$. By  Theorem~\ref{thm:main1poly}, for any $n$ there exists a hyperbolic cusp $\mathrm{C}_n$ with convex  boundary and induced metric on $\partial \mathrm{C}_n$ isometric to $m_n$.

As mentioned in the introduction, the universal cover $\widetilde{\mathrm{C}_n}$ can be isometrically embedded as a convex subset of $\H^3$ via the developing map $D$.
 The action of the fundamental group $\pi_1(\mathrm{C}_n) \cong \pi_1(T)$ on $\widetilde{\mathrm{C}_n}$ by deck transformations yields a representation $\rho: \pi_1 (T) \to \mathrm{Iso}^+(\H^3)$.
The cusp $\mathrm{C}_n$ contains a totally umbilic torus $M$ with Euclidean metric. It
 follows that the developing map sends the universal cover of $M$ to  the horosphere $H$. The group $\rho(\pi_1 (\mathrm{C}_n))=\Gamma_n$ acts on $D(\widetilde{M})$ freely with a compact orbit space. 
 The group $\Gamma_n$ is a group of parabolic isometries. 
 The surface $S_n=\partial D(\widetilde{\mathrm{C}_n})$ is convex and globally invariant under the action of $\Gamma_n$. It is easy to see that 
 $S_n$ is homeomorphic to $H$ via the central projection from the center of $D(\widetilde{M})$.
Up to rotations of the hyperbolic space, we normalize the surfaces $S_n$ in such a way that the point fixed by $\Gamma_n$ is 
$\infty$ in the half space model.
Moreover, choosing a point $x_0$ in the universal cover of the torus, up to compose by parabolic and hyperbolic isometries, we consider that the developing maps send $x_0$ onto
$(0,0,1)_E$. 

So the surfaces $S_n$ are described by horoconvex functions $u_n$. We identify $\Gamma_n$ with the corresponding lattice in $H=\R^2$. In particular, 
 $u_n(\gamma\cdot x)=u_n(x)$ for any $\gamma\in \Gamma_n$ and
$x\in \R^2$, and the normalization above says that   $u_n (0) = 0$
for any $n$. The quotient of $d_{u_n}$ by $\Gamma_n$ is isometric to $m_n$.

We will prove that $(u_n,\Gamma_n)$ converge to some $(u,\Gamma)$ and that the quotient of $d_u$ by $\Gamma$ is isometric to $m$. 
This will prove Theorem~\ref{thm:main}: the wanted cusp is the quotient by $\Gamma$ of the convex side of the horograph of $u$.

\subsection{A uniform bound on horographs}

 From Lemma~\ref{lem:diam conv}, 
there exists a uniform upper bound $\mathsf{diam}$ of all the diameters of the metrics $m_n$.

\begin{lem}\label{lem : uniform seq}
\begin{enumerate}
\item
The sequence $(u_n)_n$ is uniformly bounded.
\item There is a compact set $D\subset\R^2$ such that for any $y\in \R^2$ and any $n$, there exists $\gamma\in\Gamma_n$ with $\gamma\cdot y\in D$.
\end{enumerate}
\end{lem}

\begin{proof}
Recall that all the horographs $S_n$ of $u_n$ pass through $x_0=(0,0,1)_E$. Let  $b_n$
be the set of points on $S_n$ at distance $\leq$ $\mathsf{diam}$ from $(0,0,1)_E$ in the intrinsic metric of $S_n$. Then as the distance 
on $S_n$ is greater than the extrinsic distance of $\H^3$, all the $b_n$ are contained in a same hyperbolic ball $B$.

In the half space model, let $D$ be the projection of the ball $B$ onto the horizontal plane passing through the origin. Observe that $D=\overline{B^{\R^2}_\delta(0)}$ is a Euclidean closed ball centred at the origin of $\R^2$.
As $B$ is contained between two horospheres centred at $\infty$ (i.e. two horizontal planes)
the horofunctions $u_n$ are uniformly bounded on $D$, say $c_1<u_n<c_2$. Now by construction, for any $y\in \R^2$, there exists $\gamma\in\Gamma_n$ such that $\gamma\cdot y\in D$. Hence
$c_1<u_n(\gamma\cdot y)=u_n(y)<c_2$.
\end{proof}

%

\subsection{Convergence of groups}

\begin{lem}\label{lem : group conv}
There exists a sequence $(a_n,b_n)_n$ of generators of  $\Gamma_n$
that converges in $\R^2$ (up to extract a subsequence)
to  two linearly independent non-zero vectors $a$ and $b$.
 (Here we identify an element of $\gamma$ of $\Gamma_n$
with the vector $\gamma \cdot 0$ of $\R^2$.)
\end{lem}

\begin{proof}

Let us choose generators $(a_n,b_n)$ of $\Gamma_n$ which are contained in $\overline{B^{\R^2}_{3\delta}(0)}$, that is possible by the second item of Lemma~\ref{lem : uniform seq}.
Since $\overline{B^{\R^2}_{3\delta}(0)}$ is compact, up to take a subsequence we get the existence of two vectors $a,b\in\R^2$ such that $a_n\to a$ and $b_n\to b$ as $n\to\infty$.

Suppose that either one of the vectors $a$ or $b$ is zero, or that they are parallel. By continuity we necessarily have that the area of the parallelogram with side $a_n$ and $b_n$ tends to zero, as $n\to\infty$. In turn, this means that the area of a fundamental domain of $\R^2$ for the action of $\Gamma_n$ tend to zero as $n\to\infty$. Applying Lemma~\ref{lem:bilip} with $K=\overline{B^{\R^2}_{3\delta}(0)}$, by Proposition 3.1.4 in \cite{AT04}, the two dimensional Hausdorff measure of $m_n$
tends to zero, thus contradicting Theorem 2.4.
\end{proof}

\subsection{Construction of the solution}

By Corollary~\ref{cor : equilip} and Lemma~\ref{lem : uniform seq}, up to extract a subsequence, the sequence $(u_n)$ converges to a horoconvex function $u$, uniformly on any compact set.

Let $a,b\in\R^2$ given by Lemma~\ref{lem : group conv}, and define $\Gamma\subset \operatorname{Iso}(\R^2)$ as the direct product of $\left\langle a\right\rangle$ and $\left\langle b\right\rangle$. Since $a$ and $b$ are linearly independent vectors, $\R^2/\Gamma$ is a torus.

\begin{lem}\label{lem : well-def}
The function $u$ is $\Gamma$-invariant.
\end{lem}
\begin{proof}
Let $y\in \R^2$ and $\gamma\in\Gamma$ such that $\gamma y=y+ka+k'b$, where $k,k'\in \Z$.
Then, for every $\epsilon>0$
\begin{align}
|u(\gamma y)-u(y)|
&=|u(y+ka+k'b)-u(y)|\nonumber\\
&\leq |u(y+ka+k'b)-u_n(y+ka+k'b)|\label{1}
\\
&+|u_n(y+ka+k'b)-u_n(y+ka_n+k'b_n)|\label{2}
\\
&+|u_n(y+ka_n+k'b_n)-u_n(y)|\label{3}
\\
&+|u_n(y)-u(y)|<\epsilon\label{4} 
\end{align}
for $n$ large enough. In fact $k(a-a_n)+k'(b-b_n)\to0$ as $n\to\infty$, and as the $u_n$ are equi-Lipschitz on a sufficiently large compact set
 the absolute value at line \eqref{2} is smaller than $\epsilon/4$ for $n$ large enough. Moreover, the absolute value at line \eqref{3} is zero for every $n$ by the $\Gamma_n$-invariance of $u_n$, and the absolute value at lines \eqref{1} and \eqref{4} are smaller then $\epsilon/4$ for $n$ large enough by the uniform convergence of the $u_n$. Since $\epsilon>0$ is arbitrary, this concludes the proof.
\end{proof}

\subsection{Convergence of metrics}\label{sec:conv met}

In the preceding section, we have constructed a pair $(u,\Gamma)$. It remains to check that 
the induced metric $m_u$ on $(\R^ 2,d_u)/\Gamma$ is isometric  to $(T,m)$. Basically, one has to 
check that if the sequence $(u_n)$ converges, then the  sequence of induced metric converges. In the remaining of this section we will prove the following result, that ends the proof of the theorem, because on compact metric spaces, uniform convergence imply Gromov--Hausdorff convergence, and the Gromov--Hausdorff limit is unique.

\begin{prp}\label{prop:conv metric}
The sequence $(m_n)$ uniformly converges to $m$.
\end{prp}

\begin{lem}\label{lem: E(K)}
Let $K\subset \R^2$ be compact. Then 
$$E(K)=\operatorname{closure}\left(\cup_{x\in K}\cup_n \{y | d_{u_n}(x,y)\leq \mathsf{diam}\}\right) $$
is compact.
\end{lem}
\begin{proof}
We have $$d_{\H^3}\left(\left(x,e^{-u_n(x)}\right),\left(y,e^{-u_n(y)}\right)\right)\leq d_{u_n}(x,y)\leq \mathsf{diam}~,$$ 
but \cite[4.6.1]{rat}
$$\cosh d_{\H^3}\left(\left(x,e^{-u_n(x)}\right),\left(y,e^{-u_n(y)}\right)\right)=1+\frac{\|x-y\|^2}{2e^{-u_n(x)}e^{-u_n(y)}}$$ 
so
$$\|x-y\|\leq \sqrt{2}e^{-c}\sqrt{\cosh (\mathsf{diam})-1} $$
where $c$ is the uniform lower bound of the $u_n$.
The result follows because 
 $K$ is compact.

\end{proof}

\begin{cor}\label{lem:bound euclidien}
Let $K\subset \R^2$ be  compact. 
 Let $\mathcal{L}(K)$ be the set of shortest paths for any
$d_n$ between points of $K$ (we don't ask the shortest path to be contained in $K$).
Then there exists a constant $\alpha$ such that  $\forall c\in\mathcal{L}(K)$, $L_{\R^2}(c)\leq \alpha$.
\end{cor}
\begin{proof}
By Lemma~\ref{lem:bilip} applied to $K$, for any $x,y\in K$, 
$$d_{u_n}(x,y)\leq \lambda_2(K) \operatorname{diam}(K)~.$$ Let $c$ be 
a shortest path for $u_n$ between $x$ and $y$, so
that $L_{u_n}(c)=d_{u_n}(x,y)$. 

The shortest path is contained in the compact set $E(K)$ given by Lemma~\ref{lem: E(K)}, so by Lemma~\ref{lem:bilip} again,
but applied to $E(K)$, $L_{\R^2}(c)\leq \frac{\lambda_2(K)}{\lambda_1(E(K))}\operatorname{diam}(K)$.
\end{proof}

\begin{lem}\label{lem:conv uni}
Let $K\subset \R^2 $ be compact. Then 
$d_{u_n}$ uniformly converge to $d_u$ on $K$.
\end{lem}
\begin{proof}
Let $\epsilon >0$. As $u_n$ uniformly converge to $u$, 
for $n$ sufficiently large, $u \leq u_n + \epsilon$, so
 by Lemma~\ref{cor:BF2}
$$d_u \leq d_{u_n+\epsilon}+4\epsilon~. $$

Let $x,y\in K$ and let $c$ be a shortest path
for $u_n$ between $x$ and $y$. Then 
$$d_{u_n+\epsilon}(x,y)\leq L_{u_n+\epsilon}(c) $$
and from \eqref{eq:length}, the fact that the $u_n$ are uniformly bounded and Corollary~\ref{lem:bound euclidien},
there exists a constant $\beta$ depending only on $K$
such that
$$L_{u_n+\epsilon}(c)\leq L_{u_n}(c)+\sqrt{|e^{2\epsilon}-1|}\beta(K)$$
and as $L_{u_n}(c)=d_{u_n}(x,y)$ we obtain
$$d_u(x,y)-d_{u_n}(x,y) \leq  4 \epsilon+\sqrt{|e^{2\epsilon}-1|}\beta(K)~.$$

Exchanging the role of $u$ and $u_n$, we finally obtain
$$|d_u(x,y)-d_{u_n}(x,y)|\leq  4 \epsilon+\sqrt{|e^{2\epsilon}-1|}\beta(K)~.$$

\end{proof}

\begin{rem}
A result of A.D.~Alexandrov gives a weaker convergence of the induced metrics
for any convex surfaces converging in the Hausdorff sense to a convex surface
in the hyperbolic space, see \cite{slu}.
\end{rem}

Let $\tilde{\phi}_n$ be the linear isomorphism sending 
$a$ to $a_n$ and $b$ to $b_n$.
Hence clearly, for any $\gamma\in \Gamma$,
\begin{equation*}\label{eq:equivariance}\tilde{\phi}_n
(\gamma\cdot x)=\tilde{\phi}_n(\gamma)\cdot \tilde{\phi}_n(x)\end{equation*}
where $\gamma$ is considered as a vector of $\R^2$.
The  map
$\tilde{\phi}_n$
 descends to a homeomorphism $\phi_n$ between $\R^2/\Gamma$ and $\R^2/\Gamma_n$.

\begin{lem}\label{lem:conv phi}
Let $K\subset \mathbb{R}^2$ a compact set. Then on $K$,
$$x\mapsto d_{u_n}(\tilde{\phi}_n(x),x) $$
uniformly converge to zero and 
$d_{u_n}(\tilde{\phi}_n(x),\tilde{\phi}_n(y)) $
uniformly converge to $d_{u}(x,y)$.
\end{lem} 
\begin{proof}
 $\tilde{\phi}_n$ converge to the identity map, uniformly on any compact for the Euclidean metric.
The first result follows from  
Lemma~\ref{lem:bilip}.
The second result follows easily from the first one, the triangle inequality and 
Lemma~\ref{lem:conv uni}.
\end{proof}

\begin{lem}\label{lem: def K}
There exists a compact set $K\subset \R^2$ such that for any $n$, for any $p,q$ in the torus, a lift of a shortest path for $m_n$ between $p$ and $q$ is contained in $K$.
\end{lem}

\begin{proof}
Let $K=E(D)$, where $D$ is the compact set  obtained in Lemma~\ref{lem : uniform seq}.
By definition of $D$, there exists a lift $x$ of $p$ in $D$. 
By construction of $K$, the ball for $d_n$ centred at $x$ with radius $m_n(p,q)$ is contained in $K$. 
\end{proof}

Let $C$ be the closure of $\cup_n \tilde{\phi}_n(K)$, where
$K$ is given by Lemma~\ref{lem: def K}. $C$ is a compact set.

\begin{lem}\label{lem:lift geod}
For any $\nu > 0$, if $n$ is sufficiently large,  for any $p,q\in T$, if $x$ and  $y$ are respective lifts to the set $C$ defined above, if $d_{u}(x,y)=m(p,q)$, then  
$$m_n(\phi_n(p),\phi_n(q))\leq d_{u_n}(\tilde{\phi}_n(x),\tilde{\phi}_n(y)) \leq
m_n(\phi_n(p),\phi_n(q))+\nu~. $$
\end{lem} 
\begin{proof}
As $d_{u}(x,y)=m(p,q)$, for any $\gamma\in\Gamma$, 
$$d_u(x,\gamma\cdot y) \geq d_u(x,y)~.$$
By Lemma~\ref{lem:conv phi}, uniformly on $C$, if $n$ is sufficiently large,
$$ d_{u_n}(\tilde{\phi}_n(x),\tilde{\phi}_n(\gamma)\cdot \tilde{\phi}_n(y))+\nu \geq  d_{u_n}(\tilde{\phi}_n(x),\tilde{\phi}_n(y))~.$$
The result follows because $m_n(\phi_n(p),\phi_n(q))$
is the minimum on $\Gamma$ of all the $ d_{u_n}(\tilde{\phi}_n(x),\tilde{\phi}_n(\gamma)\cdot \tilde{\phi}_n(y))$.
\end{proof} 

\begin{proof}[Proof of Proposition~\ref{prop:conv metric}]

Let $\epsilon >0$.
For any $p,q\in T$, with the notations of Lemma~\ref{lem:lift geod},
 for $n$ large enough,
 $$m_n(\phi_n(p),\phi_n(q)) - m(p,q) $$
$$\leq  d_{u_n}(\tilde{\phi}_n(x),\tilde{\phi}_n(y))-d_u(x,y)~. $$

By Lemma~\ref{lem:conv phi} applied
on the compact set $C$, for $n$ sufficiently large, the last quantity above is less than $\epsilon$, independently of $x$ and $y$. For the same reasons, the same conclusion holds
for
$$ m(p,q) - m_n(\phi_n(p),\phi_n(q)) $$
$$\leq d_u(x,y)- d_{u_n}(\tilde{\phi}_n(x),\tilde{\phi}_n(y))+\nu~. $$

\end{proof}

\bibliographystyle{alpha}
\bibliography{existence}

\end{document}